\documentclass{compositio}
\renewcommand{\contentsname}{Contents\\{\footnotesize\normalfont(A table
of contents should normally not be included)}}
\usepackage{hyperref,color,amsmath,mathabx,graphicx,float,stackrel,xcolor,caption,subcaption}

\newtheorem{thm}{Theorem}[section]
\newtheorem{prop}[thm]{Proposition}

\newtheorem{lem}[thm]{Lemma}
\newtheorem{rem}[thm]{Remark}
\newtheorem{defn}[thm]{Definition}
\newtheorem{fact}[thm]{Fact}
\newtheorem*{fact*}{Fact}

\newcommand{\opcit}{{\it op.cit.\/}\ }
\newcommand{\ie}{{\it i.e.\/}\ }

\def\cE{{\mathcal E}}
\def\A{{\mathbb A}}
\def\B{{\mathbb B}}
\def\C{{\mathbb C}}
\def\M{{{\mathbb M}at}}
\def\Q{{\mathbb Q}}
\def\R{{\mathbb R}}
\def\W{{\mathbb W}}
\def\Z{{\mathbb Z}}

\def\Tr{{\rm Tr}}

\def\tr{{\rm Tr}}
\def\rep{\vartheta}

\def\qd{\,^-\!\!\!\!\!d}

\def\cH{{\mathcal H}}
\def\F{{\mathbb F}}
\def\fourier{\F}
\def\Fa{\F_{e_\R}}
\def\Fm{\F_\mu}

\def\id{{\mbox{Id}}}
\def\cP{{\mathcal P}}
\def\cS{{\mathcal S}}

\def\cE{{\mathcal E}}

\def\Spec{{\rm Spec\,}}
\def\sss{{\mathbb S}}
\def\fourier{\F}
\def\sr0{{\cS^{\rm ev}_0}}

\def\mat{{\rm Mat}}
\usepackage[all]{xy}
\usepackage{yfonts}
\def\Se{\frak{ Sets}}

\def\rep{\vartheta}

\def\rmax{{\R_+^{\rm max}}}

\def\Aut{{\rm Aut}}

\def\End{{\rm End}}

\def\GL{{\rm GL}}

\def\Hom{{\rm Hom}}

\def\sar0{{\cS_0(\A_\Q)}}

\def\cF{{\mathcal F}}
\def\cA{{\mathcal A}}

\def\Hom {{\mbox{Hom}}}

\def\nt{\N^{\times}}
\def\N{{\mathbb N}}
\def\fr{{\rm Fr}}

\def\Se{\frak{Sets}}

\def\cC{{\mathcal C}}

\def\cF{{\mathcal F}}

\def\rs{{\R_+^*}}
\def\bars{{\overline\sss}}

 \def\scal2{{\frak S}}

\def\gop{{\Gamma^{\rm op}}}

\def\qqq{\,,\,~\forall}

\begin{document}

\title{BC-system, absolute cyclotomy\\ and the quantized calculus }
\author{Alain Connes}
\email{alain@connes.org}
\address{Coll\`ege de France and IHES}
\author{Caterina Consani}
\email{cconsan1@jhu.edu}
\address{Dept. of Mathematics, Johns Hopkins University}

\dedication{
{\bf Dedicated to Dennis Sullivan on his 80th Birthday}\vspace{.15in}

\hspace*{10cm}\begin{minipage}{0.5\textwidth}
    .. La vraie jeunesse ne s'use pas.\\
On a beau l'appeler souvenir,\\
On a beau dire qu'elle dispara\^it,\\
On a beau dire .. que tout s'en va,\\
Tout ce qui est vrai reste l\` a.\\
(J. Pr\' evert)
\end{minipage}
\vspace{.1in}

\today
}
%\thanks{Thanks, note: at the end Acknows appears as well}

\begin{abstract}
We give a  short survey on several developments on the  BC-system, the adele class space of the rationals, and on the understanding of the "zeta sector" of the latter space as the Scaling Site. The new result that we present concerns the description of the BC-system as the universal Witt ring (i.e. K-theory of endomorphisms) of the "algebraic closure" of the absolute base $\sss$. In this way we attain a conceptual meaning of the BC dynamical system at the most basic algebraic level.  Furthermore, we define an invariant of Schwartz  kernels in 1 dimension and relate the Fourier transform (in 1 dimension) to its role over the algebraic closure of $\sss$. We implement this invariant  to prove that, when applied to the quantized differential of a function, it provides its Schwarzian derivative. Finally, we  survey the roles of the quantized calculus in relation to Weil's positivity, and that of spectral triples in relation to the zeros of the Riemann zeta function.
\end{abstract}

\maketitle

%\tableofcontents

%\vspace*{6pt}\tableofcontents  % for this guide only.
% A table of contents should normally not be included

\section*{Introduction}
We dedicate this paper to Dennis Sullivan,  whose  genuine love for understanding mathematics and his generosity in communicating new ideas has always been an inspirational example to us. We take this opportunity to write an overview on the actual state of our enduring interest in the Riemann zeta function. The Riemann zeta function has many intriguing manifestations in science:  our own interest was triggered  with the discovery of the surprising  relation that this function has with noncommutative geometry.
 In quantum statistical mechanics in the first place, the zeta function appears as the partition function of a dynamical system  determined by the analysis of the Hecke algebra of the affine group of rational numbers \cite{BC}. This result  leads to the noncommutative space of adele classes of the rationals \cite{Co-zeta}. The impact in number theory is the spectral realization of the zeros of the Riemann zeta function  and the geometric understanding of the Riemann-Weil explicit formula   as a trace formula \cite{Co-zeta, Meyer}. Moreover, the quantized calculus in noncommutative geometry, jointly with the theory of prolate spheroidal functions in analysis  provides, by means of a semilocal trace formula computation, a conceptual reason for the positivity of  Weil's functional \cite{quasiinner}. The understanding of the adele class space of the rationals from a more classical geometric standpoint is provided by  the theory of Grothendieck toposes: indeed, in many cases the space of the points of these toposes is noncommutative. The ``zeta sector'' of the adele class space of the rationals is described precisely by the set of points of the Scaling Site \cite{CCscal1}. This  result led us  to a parallel and independent investigation of  the algebraic landscape of semirings of characteristic one, where each integer acts by an endomorphism, thus generalizing the Frobenius operator on geometries in finite characteristic, and where categorically, one only adds one more  (prime) ``field'':  the Boolean semifield $\mathbb B$ \cite{CC4}. However, and in spite of its elementary definition, the Boolean   $\mathbb B$ cannot qualify as realizing  the original dream of J. Tits in his search for a basic algebraic structure rooting fundamental examples of combinatorial geometries, neither as an algebraic incarnation of Waldhausen ``initial ring''. The weakness of the classical algebraic approach is in part due to the inherent deficiency of set-theory when compared to the more flexible categorical counterpart. By following the simple idea that an abelian group $A$ is entirely encoded by the  covariant functor $HA$ that assigns   to a pointed set $X$  the pointed set of $A$-valued ``divisors'' on $X$, one quickly realizes that the functorial viewpoint is a very natural and versatile generalization of the original set-theoretical notion of abelian group. This idea has  led us to develop algebraic geometry over a ``base'' $\sss$ that is the spherical counterpart of the multiplicative monoid $\F_1=\{0,1\}$ and the categorical backbone of the sphere spectrum in homotopy theory \cite{DGM}. Rings over this base are the $\Gamma$-rings of G. Segal and the simplest of them is  $\sss$, namely the identity endofunctor on the category of pointed sets.  Ordinary  rings become $\Gamma$-rings through the Eilenberg-MacLane functor $H$.   In \cite{CCprel} we gave an  arithmetic application of these ideas by extending, at archimedean infinity, the structure sheaf of the algebraic spectrum of the ring of integers    as a sheaf of $\Gamma$-rings, and more precisely as a subsheaf of the constant sheaf $H\Q$. One instance of the relevance of the stalk at the  archimedean place of this compactification (the  Arakelov one point compactification) is 
exemplified through its relation with the Gromov norm \cite{CCgromov}. The choice of this new base  has also the further advantage to provide the right framework for Hochschild and cyclic homologies, since simplicial $\Gamma$-sets (\ie $\Gamma$-spaces) are well understood and  homotopy theory over them becomes, by means of the Dold-Kan correspondence,   homological algebra  taking place over $\sss$ \cite{CCAtiyah}.  \newline 
In Section \ref{bcrole} we give a  short survey of all these developments centered on the key role played by the  BC-system and  the adele class space, and on the understanding of the ``zeta sector'' of the latter space as the Scaling Site. The new result that we present in this context is introduced in Section \ref{bcbarS}, where we describe the BC-system as the Witt ring (\ie K-theory of endomorphisms) of the ``algebraic closure'' of $\sss$. In this way we obtain a conceptual meaning of the BC dynamical system no longer from analysis (in terms of quantum statistical mechanics) but at the most basic algebraic level. The algebraic closure   $\overline{\sss}$ is defined  by adjoining to $\sss$ all (abstract) roots of unity, and the relation between its  Witt ring and the BC-system suggests to perform the following  two steps
\begin{enumerate}
\item Determine the extension of scalars $\overline{\Spec \Z }\times_\sss \overline{\sss}$.
\item Define an appropriate   De Rham-Witt complex for  $\overline{\Spec \Z }\times_\sss \overline{\sss}$.
\end{enumerate}
An educated guess on  $\overline{\Spec \Z }\times_\sss \overline{\sss}$  suggests  that this space ought to involve algebraically the cyclotomic extension  of the field of rational numbers. 
The De Rham-Witt complex of  $\overline{\Spec \Z }\times_\sss \overline{\sss}$ should mainly provide a  strengthening of the link between two worlds:  on one side (say on the left) the classical world of Arakelov geometry now enriched over $\Gamma$-rings,  while on the other side (the right) the analytic framework of noncommutative geometry stemming from the BC-system and directly related to the understanding of the zeros of the Riemann zeta function. These two worlds are, a priori, quite different in nature. Homological algebra  over $\Gamma$-rings is, through the Dold-Kan correspondence, naturally encoded by the homotopy theory of $\Gamma$-spaces, so that the world on the left  is that of homotopy theory, spectra, animas.... The world on the right side instead, is that of analysis, Hilbert space operators, the quantum...
One fundamental relation between these two worlds is the assembly map \cite{BCH} which associates to $K$-homology classes of the universal proper quotient,  classes in the $K$-theory of the reduced $C^*$-algebra of a group.  This  creates a bridge, of index theoretic nature, between the world of homotopy theory and the world of analysis where $K$-theory of $C^*$-algebras plays a key role.  More precisely,  the assembly map relates together two ways of effecting the quotient of a space by a group action. On the left world one sees an homotopy quotient as a special case of a homotopy colimit, while on the right world one effects a cross product which is a special case of a general principle in noncommutative geometry of encoding difficult quotients (such as leaf spaces of foliations) by noncommutative algebras. It is also worth noticing that aside from quotient spaces,  these tricky spaces also appear naturally as sets of points of a topos. For example,  to a small category $\cC$  one may associate the presheaf topos $\widehat \cC$ of contravariant functors from $\cC$ to the category of sets. In general, the nature of the space of points of the topos $\widehat \cC$ is as delicate  as that of a quotient space, and one may either use as a substitute the classifying space $B\cC$  in the left world or view such spaces as noncommutative spaces, if one prefers to work in the right world.\newline
In Section \ref{sectanalysis} we  describe the role of the 1-dimensional quantized calculus in relation to Weil's positivity. A central role is played both by the unitary obtained by composing Fourier transform with inversion, and by its  quantized logarithmic derivative. An elementary lemma only meaningful in  quantized calculus (called the Main Lemma in this paper)  gives the conceptual reason to expect  Weil's positivity. The fact that the hypothesis of this lemma is only verified up to an infinitesimal prevents one to conclude immediately that positivity holds. In \cite{weilpos} we showed that positivity can still be obtained, for a single archimedean place, by treating separately this infinitesimal. Moreover,  in the semilocal case, (i.e. when  finitely many places, including the archimedean one, are involved), the same  infinitesimal property  continues to hold \cite{quasiinner}, and this fact opens the way for a strategy toward RH. In  \S \ref{Schw&S} we define an invariant of Schwartz  kernels in 1 dimension and relate the Fourier transform (in 1 dimension) to its role over $\bars$ (see \S \ref{sec:2.3}). Then, we implement this invariant  to prove that, when applied to the quantized differential of a function, it delivers its Schwarzian derivative. This shows in particular that, as emphasized in the fax of D. Sullivan reported in Figure \ref{fax} (beginning of Section \ref{sectanalysis}), the quantized differential calculus encodes in a subtle manner the conformal structure also in dimension 1, where the Riemannian point of view gives no clue. In \S \ref{sectmainlem} we state the Main Lemma in quantized calculus that yields Weil's positivity as a consequence of the triangular property of the quantized differential, and in \S \ref{semiloc} we discuss the triptych formed by Fourier, Zeta and Poisson. The quantized calculus is then  applied  in the semilocal framework (\S \ref{sectsemiloc}) and  provides, through the semilocal trace formula, both the operator theoretic formalism for the explicit formulas of Riemann-Weil and a conceptual reason for Weil's positivity. We discuss the radical of Weil's quadratic form  in \S \ref{radical} and the ``almost radical" of its restriction to an interval $[\lambda^{-1},\lambda]$   in \S \ref{spectrip}. We then use spectral triples (through Dirac operators)   to detect the zeros of the Riemann zeta function up to imaginary part $2\pi \lambda^2$. This provides the operator theoretic replacement for the Riemann-Siegel  formula in analytic number theory and the approximation to the sought for cohomology discovered in \cite{ccspectral}.

\section{The BC-system and  its role}\label{bcrole}

The  origin of the relation between noncommutative geometry and the Riemann zeta function is  a fundamental interplay between the mechanism of symmetry breaking in physics and the theory of ambiguity of E. Galois. In physics, the choice of an extremal  equilibrium state at zero temperature  breaks the symmetry of a system. On the Galois side the choice of such a state selects a group isomorphism of the abstract group $\Q/\Z$ with the group of roots of unity in $\C$. The link is established explicitly by implementing the formalism of quantum statistical mechanics \cite{BR} that encodes a quantum statistical system by a pair $(\cA,\sigma_t)$  of a $C^*$-algebra $\cA$ and a 1-parameter group of
automorphisms $\sigma:\R \to \Aut(\cA)$. The main tool  is the KMS condition that analytically encapsulates the relation existing in quantum mechanics between the Heisenberg time evolution of observables $\sigma_t(A)=\exp(itH)A \exp(-itH)$ ($H$ is the Hamiltonian of the system) and an equilibrium state $\phi$ at inverse temperature $\beta=\frac{1}{kT}$,  whose evaluation on an observable $A$ is  $\phi(A):=Z^{-1} \Tr(A \exp(-\beta H)$, where $Z=\Tr(\exp(-\beta H)$.  The precise mathematical encoding of this relation was obtained by Haag, Hugenholtz and Winnink \cite{HHW}, starting from earlier work of Kubo, Martin and Schwinger. A  way to understand the KMS$_\beta$ condition is provided by the equality
 $$
\left(\varphi(x\sigma_t(y))\right)_{t=i\beta} =\varphi(yx)
$$ 
whose heuristic meaning is that $\sigma_t$ at ${t=i\beta}$ compensates for the lack of tracial property of the state $\varphi$ by allowing one to replace $\varphi(yx)$ with $\varphi(x\sigma_t(y))$ at ${t=i\beta}$. The states fulfilling the KMS$_\beta$ condition form a (possibly empty) convex compact simplex. \newline
The specific system that exhibits the interplay  between the phenomenon of symmetry breaking in physics and the theory of ambiguity of E. Galois is the BC-system \cite{BC}. It is defined using the affine group   \[
P^+(\Q):=\Bigg\{\left(
\begin{array}{cc}
 1 & b \\
 0 & a \\
\end{array}
\right)\mid a,b\in \Q, a>0\Bigg\}.\]
The subgroup $P^+(\Z)$ of integral translations  obtained by requiring that $a,b\in \Z$ is  almost normal in $P^+(\Q)$, and this fact allows one to define a Hecke algebra $\cA$ in place of the convolution algebra of the quotient $P^+(\Q)/P^+(\Z)$. The action of $\cA$ in the Hilbert space $\ell^2(P^+(\Q)/P^+(\Z))$ plays the role of the regular representation. The significant fact here is that this representation determines a factor of type III, thus naturally endowed with a  one parameter group of automorphisms $\sigma_t$ of $\cA$   (the time evolution). The pair $(\cA,\sigma_t)$ constitutes the BC-system. Its first properties are the following:

\noindent $\blacktriangleright$~The system exhibits a phase transition with spontaneous symmetry breaking. The KMS$_\beta$ state is unique for $\beta\leq 1$. For $\beta> 1$ the
extremal KMS$_\beta$ states are parameterized by the points of  the zero-dimensional Shimura variety $Sh(\GL_1,\{\pm 1\})$.\newline
\noindent $\blacktriangleright$~The symmetries of the system are  given by the group
$\GL_1(\hat\Z)=\hat\Z^*$ of invertible elements of the profinite completion of the integers. The zero-temperature KMS states  evaluated on a natural arithmetic subalgebra of the
algebra of observables of the system take values that are algebraic
numbers and generate the maximal abelian extension $\Q^{\rm    cycl}$ of $\Q$.\newline
\noindent $\blacktriangleright$~The class field
theory isomorphism intertwines the action of the
symmetries and the Galois action on the values
of states, thus  providing a quantum
statistical mechanical reinterpretation of the explicit class field
theory of $\Q$.\newline
\noindent $\blacktriangleright$~The partition
function $Z(\beta)$ of the system is the Riemann zeta function evaluated at $\beta\in\R$.

\noindent This last property  establishes the link between noncommutative geometry and the Riemann zeta function. The algebra of the BC-system describes the quotient space $\Q^{\times}\backslash \A_f$ of finite adeles of $\Q$ acted upon by   the multiplicative group $\Q^{\times}$. When passing to the dual system using the dynamics, and combining the dual action of $\rs$ together with the symmetries $\GL_1(\hat\Z)=\hat\Z^*$ of the system,  one obtains the action of the idele class group on the adele class space $\Q^{\times}\backslash \A_\Q$.  This is the space that provides a geometric interpretation of the Riemann-Weil explicit formulas \cite{Co-zeta}.\newline
This latter result  was the starting point of a ``longue marche''  pursuing the study of the geometry of the adele class space. This space provides the spectral realization of the zeros of $L$-functions with Gr\"ossencharacter, where the Riemann zeta function is associated to  the trivial character and whose related space is the    ``zeta-sector'' $X=\Q^{\times}\backslash \A_\Q/\hat\Z^*$.   This zeta-sector   provides a Hasse-Weil formula for the Riemann zeta function using the action of $\rs$ on $X$ \cite{CC1,CC2}. In view of this result  it is clear that  $X$ may play the role of  the space of the points of the curve for function fields. The geometric structure of $X$ came with the discovery    of the ``Arithmetic Site" \cite{CCarith,CCas}: this is the presheaf topos $\widehat\nt$ dual to the multiplicative monoid of positive integers,  endowed with the structure sheaf provided by the only semifield $F$ whose multiplicative group is infinite cyclic. The geometry of the Arithmetic Site is tropical and  of characteristic one (the addition is unipotent:  $1+1=1$). The structure sheaf of  this topos is obtained by implementing the action  of the semigroup $\nt$ on the semifield $F$ by power maps $x\mapsto x^n$. It is a general fact that in characteristic one  the power maps define injective  endomorphisms of a semifield and that there exists only one semifield which is finite and  not a field, namely the Boolean semifield $\B:=\{0,1\}$.  The Arithmetic Site is defined over $\B$ (because $F$ is of characteristic one) and a key result is that the  ``zeta-sector" $X$ gets canonically identified with the set of points of the Arithmetic Site defined over the semifield $\rmax$ of tropical real numbers. This semifield appears both in tropical geometry and also  in semiclassical analysis as a limit of deformations of real numbers. One extremely convincing result of the dequantization program \cite{litvinov} is that the Fourier transform becomes the Legendre transform when taken to the classical limit. The semifield $\rmax$ is an infinite extension of $\B$ and its absolute Galois group  is determined by the power maps
\begin{equation}\label{galrmax}
\Aut_\B(\rmax)=\{\fr_\lambda\mid \lambda\in \rs\}	, \qquad \fr_\lambda(x):=x^\lambda.
\end{equation}
This group acts on the points of the Arithmetic Site defined over $\rmax$ and, under the canonical identification of these points with the ``zeta-sector" $X=\Q^{\times}\backslash \A_\Q/\hat\Z^*$, this action corresponds to the action of the idele class group. In spite of the fact that the Arithmetic Site is an object of countable nature (the semigroup $\nt$ and the semifield $F$ are countable) and hence there is no non-trivial action of $\rs$ on the topos,  $\rs$ acts meaningfully using the theory of  correspondences \cite{CCarith, CCas}.  
The extension of scalars of the Arithmetic Site  to $\rmax$  determines
the Scaling Site \cite{CCscal1}, namely the Grothendieck topos $[0,\infty)\rtimes \nt$ ($\nt$ acts by multiplication) endowed with the structure sheaf of continuous convex functions with integral slopes.   The set of points of the topos $[0,\infty)\rtimes \nt$ identifies canonically with the ``zeta sector'' $X$. The restriction of the  structure sheaf of the Scaling Site to the periodic orbits in $X$ determines,  for each prime $p$, the quotient $\rs/p^\Z$ which appears in $X$ as the counterpart of the prime-point $p$ of $\Spec \Z$. The emerging tropical structure describes an analogue of an elliptic curve and it also exhibits a few totally new features. For instance, the divisor degree on these curves is  a real number and the Riemann-Roch formula is real valued. Such real valued indices are ubiquitous in  the noncommutative geometry of foliations and the tropical geometry of the Scaling Site can be lifted in complex geometry \cite{CCscal4}. \newline 
In order to extend the geometric positivity argument used by Mattuck-Tate and Grothendieck for function fields, to the field of rational numbers and on the above geometric space one  needs to show a Riemann-Roch formula holding on the square of the Scaling Site. In this respect, the case of periodic orbits is  far too simplified since for curves one can bypass the construction of  a cohomology theory for divisors beyond $H^0$  using Serre duality as a definition of $H^1$. For surfaces, and in particular for the square of the Scaling Site, this trick handles only $H^2$ leaving  $H^1$  still out of reach. One is thus faced with the problem of developing a good cohomology theory in characteristic  one. Motivated by this application we developed a general theory of homological algebra for the (non abelian) category of $\mathbb B$-modules  \cite{CCscal3}, however the lack of the additive inverse makes the elimination of certain technical difficulties apparently quite hard. While trying to by-pass this issue,  we were lead to investigate a more fundamental base for algebraic manipulations, which is, as explained in the introduction, independent of the choice of a characteristic. The main reason for our turn of interests toward this new base is that it is the most natural one for Hochschild and cyclic homology theories.   %The article \cite{Dundas} provides a careful analysis of the issue inherent to ``group completion'' in this framework.  
In the next section we show that the fundamental basis $\sss$ provides the conceptual interpretation of the BC-system as the Witt construction over the algebraic closure of $\sss$.

\section{The conceptual meaning of the BC-system}\label{bcbarS}

 The convolution algebra of the quotient $P^+(\Q)/P^+(\Z)$ has an integral model \cite{CCarbc} \S 3,  given by the Hecke algebra ${\mathcal H}_\Z=\Z[\Q/\Z]\rtimes\N$. The ring endomorphisms $\sigma_n(e(r))=e(nr)$, $n\in\N$ act on the canonical generators of the group ring $e(r)\in \Z[\Q/\Z]$,  $r\in \Q/\Z$. There are  natural quasi-inverse linear maps
  \begin{equation*}%\label{newrhos0}
\tilde\rho_n: \Z[\Q/\Z] \to \Z[\Q/\Z]\,, \qquad \tilde\rho_n(e(\gamma))=
\sum_{n\gamma'=\gamma}e(\gamma').
\end{equation*}
These two operators are used both in the definition of the crossed product $\Z[\Q/\Z]\rtimes\N$ and in the presentation of the algebra.\newline
There is a striking analogy between  the algebraic rules fulfilled by the pair  $\{\sigma_n, \tilde\rho_n\}$ and the relations fulfilled, in the global Witt construction, by the Frobenius and  Verschiebung operators. The invariant part
 of the group ring $\Z[\Q/\Z]$ for the action of the  group  $\Aut(\Q/\Z)= \widehat \Z^*$ is described in terms of  Almkvist's ring of endomorphisms $\W_0(\sss)$ as follows

\begin{thm}(\cite{CCAtiyah} Theorem 2.3)
\label{w0s} The ring $\W_0(\sss)$ is canonically isomorphic to the invariant part
 of the group ring $\Z[\Q/\Z]$ for the action of the  group  $\Aut(\Q/\Z)= \widehat \Z^*$. 
\end{thm}

In this section we extend Theorem \ref{w0s} by showing a natural isomorphism of the group ring $\Z[\Q/\Z]$ with the ring $\W_0(\overline\sss)$, where  $\overline\sss$ denotes the monoid $\sss$-algebra
$\sss[M]$ of the  multiplicative pointed monoid $M=(\Q/\Z)_+$, with elements the base point $*=0$ and  the $e(r)$'s for $r\in \Q/\Z$. The multiplication in $M$ is defined by: $e(r) e(s)=e(r+s)$
$\forall~r,s\in \Q/\Z$. The  functor $\overline\sss: \gop \longrightarrow \Se_*$,  is defined by $\overline\sss[X]=X\wedge M$, where the monoid structure in $M$ yields the algebra structure $\overline\sss[X]\wedge \overline\sss[Y]\to \overline\sss[X\wedge Y]$.   

\subsection{Endomorphisms and matrices}\label{sec2.1}
In \cite{CCAtiyah} we considered the class of $\sss$-modules of the form $\sss[F]=\sss\wedge F$, where $F$ is   a finite object of the category $\Se_*$ of  pointed sets. As a  functor $\sss[F]:\gop \longrightarrow \Se_*$   associates to a finite pointed set $X$ the smash product $\sss[F](X):= F\wedge X$ and to a  map of finite pointed sets $g:X\to Y$ the map $\sss[F](g):=\id \wedge g$. An endomorphism  of $\sss[F]$ is a natural transformation.

\begin{lem}\label{endosm} Let $F,F'$ be two finite objects in  $\Se_*$. The map
\[
\Hom_\sss(\sss[F],\sss[F']) \to \Hom_{\Se_*}( F, F') \quad \phi\mapsto \phi(1_+),
\]
where $\phi(1_+)$ denotes the restriction of $\phi$ to $1_+=\{0, 1\}$ is a bijection of sets. The inverse map is  \[\Hom_{\Se_*}( F, F')\to \Hom_\sss(\sss[F],\sss[F']), \quad  \psi\mapsto \tilde \psi=\id\wedge \psi,\] where   $\tilde \psi(X)= \id_X\wedge \psi: X\wedge F\to X\wedge F'$.
\end{lem}
\proof Let $\phi\in \Hom_\sss(\sss[F],\sss[F'])$ and  $X$ a finite pointed set. An element $ y\in \sss[F](X)=F\wedge X$, $y\neq *$, is determined by a pair $y=(f,x)\in F\times X$, and there exists a (unique) map of pointed sets $g:1_+\to X$  with $g(1)=x$. By the naturality of the transformation $\phi$ one has:   $\phi\circ \sss[F](g)=\sss[F](g)\circ \phi$. This shows that $\phi$ is uniquely determined by its restriction $\phi(1_+)$ on $\sss[F](1_+)=F$,  with $\phi(1_+)\in \Hom_{\Se_*}( F, F')$. Conversely, given $\psi\in \Hom_{\Se_*}( F, F')$ one associates to it the natural transformation $\tilde \psi: \sss[F] \to \sss[F']$  that maps a finite pointed set $X$ to the map $\id_X\wedge \psi: \sss[F]\to \sss[F']$. It is immediate to verify that the two  maps   are inverse of each other. \endproof 

In the following part we shall consider endomorphisms of $\bars$-modules of the form $\bars[F]=\bars\wedge F$, with $F$ a finite pointed set. For $n\in \N$, $n_+:=\{0,1,\ldots, n-1,n\}$.

\begin{defn} Let $\mat_n^R(\bars)$ be the multiplicative pointed monoid of $n\times n$ matrices  with entries  in the multiplicative monoid $\bars(1_+)=M=(\Q/\Z)_+$, which have  only one non-zero (\ie not equal to the base point $*$) entry in each column. 
\end{defn}

Given $\mu=(\mu_{ij})\in \mat_n^R(\bars)$ one defines a map of pointed sets  by setting
\begin{equation}\label{endm}
	\rho(\mu):M\wedge n_+\to M\wedge n_+ \quad \rho(\mu)(\alpha,j):=\begin{cases}* & {\rm if}~ \mu_{ij}=*\quad \forall i \\ (\mu_{ij}\alpha,i)& {\rm if}~ \mu_{ij}\neq*.\end{cases}
\end{equation}
 Note that for $\mu\in \mat_n^R(\bars)$, there exists, for a given $j$, at most one $i\in \{1,\ldots, n\}$ with $\mu_{ij}\neq*$.
 
\begin{prop}\label{endosm1}  With the notations of Lemma \ref{endosm}, the map 
\[
\tilde\rho:\mat_n^R(\bars)\to \End_\bars(\bars[n_+])\quad \tilde\rho(\mu):=\widetilde{\rho(\mu)}
\]
is an isomorphism of multiplicative pointed monoids.	
\end{prop}
\proof~With $\mu=(\mu_{ij})\in \mat_n^R(\bars)$,  $\widetilde{\rho(\mu)}$ defines a natural transformation 
$$\widetilde{\rho(\mu)}: \bars[F] \to \bars[F],\quad \widetilde{\rho(\mu)}(X)= \id_X\wedge \rho(\mu): X\wedge M\wedge n_+\to X\wedge M\wedge n_+$$  
that commutes with the action of $M$. Thus it determines an endomorphism $\widetilde{\rho(\mu)}\in \End_\bars(\bars[F])$. Let $\mu, \mu'\in \mat_n^R(\bars)$: their product  is given by 
$$
(\mu \mu')_{ik}=\begin{cases}\mu_{ij}\mu'_{jk} & {\rm if}~ \exists j ~{\rm s.t.}~ \mu_{ij}\neq *~{\rm and}~   \mu'_{jk}\neq * \\ * & {\rm otherwise.}\end{cases}
$$
By applying \eqref{endm} one gets:  $\rho(\mu\mu')=\rho(\mu)\circ \rho(\mu')$, since $\rho(\mu)\circ \rho(\mu')(\alpha,k)\neq *$ if and only if there exist $j$ with $\mu'_{jk}\neq*$ and $i$ with $\mu_{ij}\neq*$. In that case, one has: $\rho(\mu)\circ \rho(\mu')(\alpha,k)=(\mu_{ij}\mu'_{jk}\alpha,i)=((\mu \mu')_{ik}\alpha,i)$. This shows that  $\tilde\rho$ is a multiplicative map. It is injective by construction. Next we show that it is also surjective. Let $\phi\in  \End_\bars(\bars[F])$. Then by Lemma \ref{endosm}  $\phi=\tilde \psi$ where $\psi$  is the restriction $\phi(1_+)$. This restriction commutes with the action of $M$ on $(\bars[F])(1_+)=F\wedge M$ and thus it is given by a matrix $\rho(\mu)$ acting as in \eqref{endm}. \endproof 
  
  For a given   $\sss$-algebra $A$,  we denote by $\M_n^{R}(A)$  the $\sss$-algebra  of matrices over $A$ defined in \cite{DGM} (\S 2.1.4.1, example 2.1.4.3, 6). 
   Note that, up to transposition,  there are two equivalent definitions for such matrices:  we let $\M_n^{L}(A)$ be the functor ($\sss$-algebra) from finite pointed sets to pointed sets that maps a finite pointed set $X$ to the set of $n\times n$ matrices of elements of $A(X)$ with only one non-zero entry in each row. Similarly,   $\M_n^{R}(A)$ is the functor mapping a finite set $X$ to the set of $n\times n$ matrices of elements of $A(X)$ with only one non-zero entry in each column.\newline
Next proposition shows that one can define a bimodule $\M_n(A)$ over these two $\sss$-algebras as the functor  from finite pointed sets to pointed sets mapping $X$ to the set of $n\times n$ matrices of elements of $A(X)$ with no restriction on the matrix entries. The proposition is in fact a special case of the composition law for $\sss$-algebras viewed as endofunctors.
 \begin{prop}\label{endosm2} Let $A$ be an $\sss$-algebra. The following facts hold 
 \begin{enumerate}
  \item[(i)]~The action of $\M_n^{L}(A)$ on   $\M_n(A)$  by left multiplication 
  $$
  \M_n^{L}(A)(X)\times \M_n(A)(Y)\to \M_n(A)(X\wedge Y)
  $$
  turns $\M_n(A)$ into a left module over  $\M_n^{L}(A)$.	
  \item[(ii)]~The action of $\M_n^{R}(A)$ on   $\M_n(A)$  by right multiplication 
  $$
  \M_n(A)(X)\times \M_n(A)^{R}(Y)\to \M_n(A)(X\wedge Y)
  $$
  turns $\M_n(A)$ into a right module over  $\M_n^{R}(A)$.	
  \end{enumerate}
  \end{prop}
  \proof The proof is the same as the one in \cite{DGM}: one simply needs to check that the product in the $\sss$-algebra $A$ determines a well defined product of matrices. To this end, the point is that the sum involved in determining the matrix element at position $(i,j)$  is obtained  from a row by column  product of two matrices  that only contain one non-zero term. This fact holds as long as either the rows of one matrix or the columns of the other one contain only one non-zero element: this is the case in (i) and (ii).\endproof 
  
 For $A=\bars$ and $X=1_+$ there is an isomorphism of pointed monoids\footnote{Note  that the $\sss$-algebra $\M_n^{R}(\bars)$ is not the same as the spherical algebra of the monoid $\mat_n^R(\bars)$} $\M_n^{R}(\bars)(1_+)=\mat_n^R(\bars)$. Moreover, the set of matrices $\mat_n(\bars)=M_n(\sss[M])$ ($M=(\Q/\Z)_+$) coincide with $\M_n(\bars)(1_+)$. By Proposition \ref{endosm2}, they form a bimodule with right and left actions provided by $\mat_n^R(\bars)$ acting on the right of $\mat_n(\bars)$ by matrix multiplication and by $\mat_n^L(\bars)=\M_n^{L}(\bars)(1_+)$ acting similarly on the left.  The role of the bimodule $\mat_n(\bars)=\M_n(\bars)(1_+)$ is to encode similarities. \newline 
 Given a field $k$ and the associated $\sss$-algebra $Hk$, morphisms of $\sss$-algebras $\bars=\sss[M]\to Hk$ correspond bijectively to (multiplicative) monoid homomorphisms $M\to k$ (\cite{schemeF1} Proposition 2.2 $(i)$). In particular, to an injective morphism $M\to k$ corresponds an extension of $\bars$ by the field $k$. In view of this fact, we introduce the following
 
 \begin{defn} \label{invert} An element $\alpha\in \mat_n(\bars)=M_n(\sss[M])$ is invertible if and only if the matrix $\alpha\in M_n(k)$ is invertible in all field extensions $k$ of $\bars$.	
 \end{defn}
 
   Matrix similarity in $\mat_n(\bars)$ is stable by taking powers, as illustrated by the following
  
 \begin{lem}\label{equiv} Let $\alpha\in\mat_n(\bars)$,   $\mu\in\mat_n^R(\bars)$, $\gamma\in\mat_n^L(\bars)$, such that $\gamma \alpha=\alpha\mu$. Then one has $\gamma^k \alpha=\alpha\mu^k$ for all $k\in \N$.	
\end{lem}
\proof One has
$\gamma^2 \alpha= \gamma \alpha\mu=\alpha\mu^2$, and by induction on $k$ one derives  $ \gamma^k \alpha=\alpha\mu^k$.\endproof

In view of Proposition \ref{endosm1}, it is equivalent to  consider endomorphisms $T\in \End_\bars(\bars[F])$ ($F$ finite pointed set) of $\bars$-modules $E=\bars[F]$,  or  matrices $\mu\in \mat_*^R(\bars)$, where $*$ is the  integer recording the cardinality of  the complement of the base point in $F$.
 One defines the notion of  invariant (of endomorphisms)   as follows
   	
 \begin{defn} \label{definv} An invariant  is a map \[\chi:\mat_*^R(\bars)\to R\]  to a commutative ring $R$  that  satisfies the following conditions:
\begin{enumerate}
\item $\chi(E,T)=\chi(T(E),T)$
\item $\chi((E_1\vee E_2, T_1\vee T_2)=\chi(E_1,T_1)+\chi(E_2,T_2)$,  $
\chi(E_1\wedge E_2,T_1\wedge T_2)=\chi(E_1,T_1)\chi(E_2,T_2)
 $,  where the smash product is taken over $\bars$
\item $\chi$ is invariant under similarity, \ie $\chi(\gamma)=\chi(\mu)$ if $\gamma \alpha=\alpha\mu$ for an invertible matrix $\alpha\in \mat_*(\bars)$.
\end{enumerate}
\end{defn}
Condition (i) is the same as in Definition 2.2 of \cite{CCAtiyah}, and has the role to mod out the zero endomorphisms. The second condition implements the ring structure. Finally,  (iii) realizes invariance under similarity. 

\subsection{Construction of the universal invariant}
We shall define the universal invariant of endomorphisms after applying the extension of scalars from $\bars$ to the maximal cyclotomic extension of $\Q$. In that set-up  Almkivst's original result applies and associates to (square) matrices a divisor with coefficients in the multiplicative group of the field. Our result states that the divisor has coefficients in the group of roots of unity. \newline
%Alternatively one can choose a prime $\ell$ far away from a given matrix and use extension of scalars to the algebraic closure of $\F_\ell$.
Next proposition gives the construction of the invariant  of endomorphisms. We keep the same notations of \S \ref{sec2.1}, in particular  $M=(\Q/\Z)_+$  denotes the multiplicative, pointed monoid of abstract roots of unity.

\begin{prop}\label{defninv} Let $T\in\mat_n^R(\bars)$,  and   $\kappa:M\hookrightarrow k$ be an injective morphism into an algebraically closed field extension of    $\bars$ of characteristic zero.
\begin{enumerate}
\item ~The
 divisor $D$  defined by Almkvist's invariant of $\kappa(T)\in\mat_n(k)$ has coefficients in $\kappa(M^\times)$. The divisor $\tau(T):=\kappa^{-1}(D)$ with coefficients in $M^\times$ is independent of the choice of $\kappa$.
\item ~The map 
\[
\tau: \mat_n^R(\bars)\to \Z[\Q/\Z], \quad \tau(T):=\kappa^{-1}(D)
\]
defines an invariant. 
\end{enumerate}
\end{prop}
\proof (i)~Let $t=(t_{ij})\in M_n(k)$ be a matrix whose non-zero entries are roots of unity and with at most one non-zero element $t_{ij}$ in each column. We claim that the eigenvalues of $t$ are either  $0$ or roots of unity. Let $E=k^n$ be the $k$-vector space on which $t$  acts. The subspaces $E_j:=t^j(E)$ form a decreasing filtration of $E$ for which there exists a finite index $\ell$ such that $E_{\ell+1}=E_\ell$. The non-zero eigenvalues of $t$ are the same as the eigenvalues of the restriction $t_\ell$ of $t$ on $E_\ell$. We verify that the endomorphism $t_\ell$  has finite order. Indeed, let 
\begin{equation}\label{phi}
\phi:n_+\to n_+\quad  \phi(j) = \begin{cases} i & \text{if $t_{ij}\neq 0$}\\ * &\text{if  $t_{ij}=0 ~\forall i$.}
\end{cases}
\end{equation}
The range of $\phi^\ell$ labels a basis of  $E_\ell$: in this basis the matrix of $t_\ell$ describes the permutation obtained by restricting $\phi$, whose entries are in roots of unity. Such a matrix is periodic thus all of its eigenvalues are roots of unity. This shows that Almkivst's invariant of $\kappa(T)\in\mat_n(k)$, \ie  a divisor $D$ with coefficients in $k^\times$, has in fact coefficients in $\kappa(M^\times)$. Moreover one also derives that the divisor $\tau(T):=\kappa^{-1}(D)$ with coefficients in $M^\times$ is independent of the choice of $\kappa$.\newline
(ii)~The map $\tau$ fulfills the three conditions of Definition \ref{definv} since they hold true for Almkivst's invariant, in particular  the operations in condition (ii)  correspond to direct sum and tensor product of modules.
\endproof

\subsection{Completeness of the invariant $\tau$}\label{sec:2.3}
To prove that the above construction defines a universal invariant, one applies the same proof as in  Theorem 3.3 of \cite{CCAtiyah}  (in the case of endomorphisms of finite $\sss$-modules), to show the injectivity of  $\tau$. The main fact to verify  is that by implementing the (algebraic) Fourier transform one can diagonalize any  matrix in $\mat_n(\bars)$ corresponding to a permutation at the set level. We shall see (in the proof of Theorem \ref{w0s1}) that for a cycle of such permutation one can choose a basis so that the matrix of such permutation is equivalent to the cyclic permutation matrix multiplied by a root of unity. Next proposition gives the  algebraic relation between the endomorphisms determined by the cyclic permutation matrix $C(n)$ of order $n$
$$
C(n)_{ij}:=\begin{cases}1 & \text{ if ~$i=j+1 ~ (n)$} \\ 0 & \text{otherwise}\end{cases} \qqq i, j\in \{0,\ldots ,n-1\}
$$
 and the diagonal matrix $\Delta(n)$ whose entries are the full set of $n$-th roots of unity, \ie $\Delta(n)_{jj}=e(j/n)$ for $j\in \{0,\ldots ,n-1\}$. The proposition shows that there is a non-trivial algebraic relation between $\Delta(n)$ and $C(n)$, and  by Lemma \ref{equiv} the same relation holds true when arbitrary powers of these two matrices are involved. 

\begin{prop}\label{fourier} For $n\in \N$, let $\mu_n:=\{e( a/n)\mid a\in \Z/n\Z\}$ be the group of $n$-th roots of unity. 
\begin{enumerate}
\item ~The matrix $V=(V_{ij})\in \mat_n(\sss[\mu_n])$: $V_{ij}=e(i j/n)
$
is the matrix of the Fourier transform on the cyclic group $\Z/n\Z$. 
\item ~In any field extension of $\sss[\mu_n]$ one has $n\neq 0$, and  the inverse of $V$ is, up to the overall factor $n$, the matrix $W=(W_{ij})$: $W_{ij}=e(-i j/n)$.
\item ~The following relations hold
\begin{equation}\label{dv}
\Delta(n)V=VC(n), \qquad C(n)W=W \Delta(n).	
\end{equation}
\end{enumerate}
\end{prop} 
\proof (i)~It suffices to recall that the Fourier transform on the cyclic group $\Z/n\Z$ is the transformation $F$ of functions $f:\Z/n\Z\to \C$ defined by
$$
F(f)(a)=\sum \epsilon(a b/n)f(b), \quad  \epsilon(x):=\exp(-2\pi i x)\qqq x\in \R.
$$ 
(ii)~Let $k$ be a  field extension of $\sss[\mu_n]$, then $k$ contains $n$ distinct roots of unity of order $n$,  thus the characteristic of $k$ is prime to $n$. It follows that $n\neq 0$ in $k$ and the inverse of $V$ is $\frac 1n W$.\newline
(iii)~One checks \eqref{dv} by direct computation using the equality $e(x)e(y)=e(x+y)$.\endproof

We can now state and prove the main result of this section

\begin{thm}\label{w0s1} The  ring $\W_0(\bars)$ is canonically isomorphic to  the group ring $\Z[\Q/\Z]$. \newline
The  invariant $\tau: \mat_*^R(\bars)\to \Z[\Q/\Z]$ is universal and it extends the  additive invariant of Theorem \ref{w0s}.
\end{thm}
\proof Let $\chi: \mat_*^R(\bars) \to R$ be an  invariant (thus fulfilling the conditions of Definition \ref{definv}), then consider the map
$
\beta:\Q/\Z\to R$,  $\beta(r):=\chi([e(r)])$, $\forall r\in \Q/\Z
$,
where $[e(r)]$ is the endomorphism of the one dimensional module  $\bars$ given by multiplication by $e(r)$. By  Definition \ref{definv} (ii),    $\beta:\Q/\Z\to R^\times$ is a group homomorphism and hence it extends to a ring homomorphism $\beta:\Z[\Q/\Z]\to R$. Next, we show that $\chi=\beta\circ \tau$. Let $T\in \End_\bars(\bars[F])$, we prove that $\chi(T)=\beta(\tau(T))$. We identify $F=n_+$ and  let $\mu=(\mu_{ij})$ be the  matrix with $\tilde\rho(\mu)=T$ (see Proposition \ref{endosm1}). Let  $\phi:n_+\to n_+$ be the map as in \eqref{phi}.   The ranges $X_\ell$ of the powers $\phi^\ell$ form a decreasing sequence of subsets and we let $\ell$ be such that $X_\ell=X_{\ell+1}$. The matrix of the restriction of $T$ to $X_\ell$ is the matrix of the permutation obtained by restricting  $\phi$, thus using  Definition \ref{definv} (i) we can just consider the case where $\mu$ is the matrix of a permutation with entries in roots of unity. The required additivity in (ii) of Definition \ref{definv}, allows one to assume that the permutation is a cyclic permutation. Then we observe that the matrix of a cyclic permutation $\nu$ with entries roots of unity is equivalent, using a diagonal matrix whose entries are  ratios of entries of $\nu$,  to the matrix of a cyclic permutation of type $e(s)C(m)$, whose  entries are all the same root of unity $e(s)$. It then follows from Proposition  \ref{fourier} and  Definition \ref{definv} (iii) that  $\chi(T)=\chi(e(s)\Delta(m))$, and finally, using again (ii) of Definition \ref{definv}, one obtains  $\chi(T)=\beta(\tau(T))$.\endproof

\subsection{Frobenius and Verschiebung}
  The Frobenius endomorphisms and the Verschiebung maps  are  operators  in $\W_0(\sss)$ \cite{CCAtiyah}. The Frobenius ring endomorphisms $F_n$, $n\in\N$, are defined by the equality $F_n((E,T)):=(E,T^n)$ (we refer to \opcit for the notations). One easily  checks that
\begin{equation}\label{2sigma}
\tau(F_n(x))=\sigma_n(\tau(x))\qqq n\in \N, \ x\in \W_0(\bars),
\end{equation}
where  the group ring endomorphism $\sigma_n$ is defined, for each $n\in\N$, by
\begin{equation}\label{defnsig}
\sigma_n: \Z[\Q/\Z] \to \Z[\Q/\Z],\qquad \sigma_n(e(\gamma))=e(n\gamma).
\end{equation}
 The Verschiebung maps $V_n$ replace a pair $(E,T)$ by the endomorphism of the sum $\vee^n E$  that cyclicly permutes the terms and uses $T:E\to E$ to turn back from the last term to the first. The map $V_n$ is additive by construction and when applied to a one dimensional $\bars$-module $[e(a)]$ it gives  the sum of the $[e(b)]$'s where $nb=a$. One thus obtains
 \begin{equation}\label{2rho}
\tau(V_n(x))=\tilde\rho_n(\tau(x))\qqq n\in \N, \ x\in \W_0(\bars),
\end{equation}
 where $\tilde\rho_n$, $n\in\N$, is defined by
\begin{equation}\label{newrhos}
\tilde\rho_n: \Z[\Q/\Z] \to \Z[\Q/\Z], \qquad \tilde\rho_n(e(\gamma))=
\sum_{n\gamma'=\gamma}e(\gamma').
\end{equation}
Then, in analogy to and generalizing  what stated in \cite{CCAtiyah} \S 2.2, one derives the following

\begin{thm} \label{rabi10} The correspondences $\sigma_n\to F_n$,
$\tilde\rho_n\to V_n$  determine a canonical isomorphism of the integral BC-system (i.e. the Hecke algebra $\cH_\Z= \Z[\Q/\Z]\rtimes\N$)  with the Witt ring $\W_0(\bars)$ endowed with the Frobenius and Verschiebung maps. 
\end{thm}

\section{Analytic approach and noncommutative geometry}\label{sectanalysis}

On September 27 1993,  Dennis Sullivan sent to the first author a fax, reproduced in Figure \ref{fax},  that sets the scene of the interactions between noncommutative geometry and geometry of manifolds\footnote{A joint paper (with N. Teleman) appeared in Topology in 1994 (\cite{CST})}. In \cite{Crh}, the account of the analytic approach based on \cite{Co-zeta} was reduced to the minimum. This section is dedicated to explain our recent results on this analytic approach. Two  main tools in noncommutative geometry play a key role here, they are:
\begin{itemize}
\item {\bf The quantized calculus}	\item {\bf The notion of spectral triple}
\end{itemize}
The quantized calculus is applied in the semilocal framework and it provides, through the semilocal trace formula, both the operator theoretic formalism for the explicit formulas of Riemann-Weil and a conceptual reason for Weil's positivity (\S\S \ref{sectmainlem}, \ref{semiloc}, \ref{sectsemiloc}).\newline
Spectral triples (through Dirac operators) together with the understanding of the radical of the Weil quadratic form restricted to an interval $[\lambda^{-1},\lambda]$ using prolate functions, and the implementation of the map $\cE(f)(x):=x^{1/2}\sum_1^\lambda f(nx)$, allow one to detect the zeros of the Riemann zeta function up to imaginary part $2\pi \lambda^2$,  thus providing the operator theoretic replacement for the Riemann-Siegel  formula in analytic number theory (\S\S \ref{radical},  \ref{spectrip}).\newline
Both notions make essential use of operators in Hilbert space and of the following  dictionary 
\begin{figure}[H]
\centering
\begin{subfigure}{.4\textwidth}
  \centering
  \includegraphics[width=.75\linewidth]{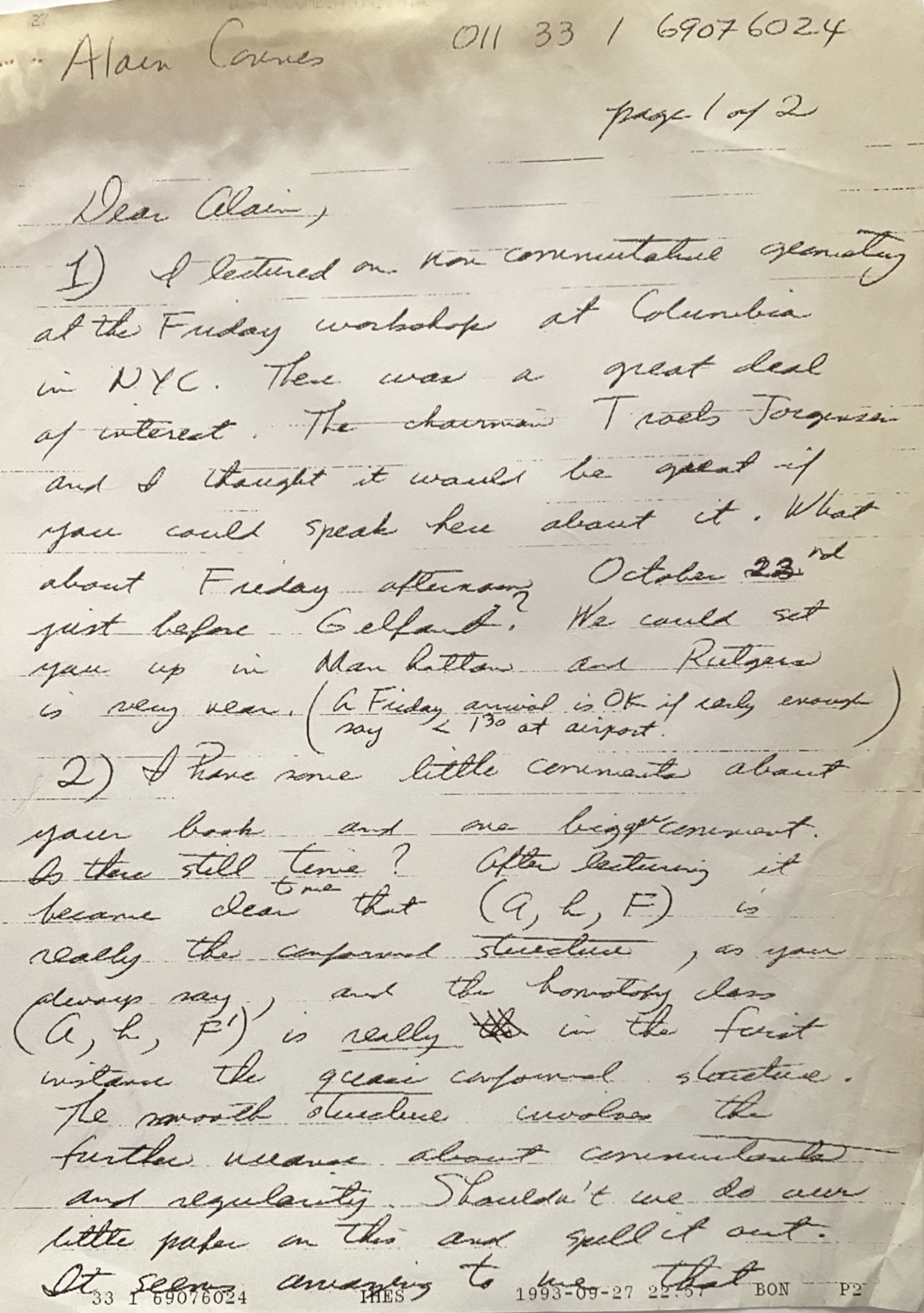}
  \caption{Page 1}
  %\label{fig:sub1}
\end{subfigure}%
\begin{subfigure}{.4\textwidth}
  \centering
  \includegraphics[width=.75\linewidth]{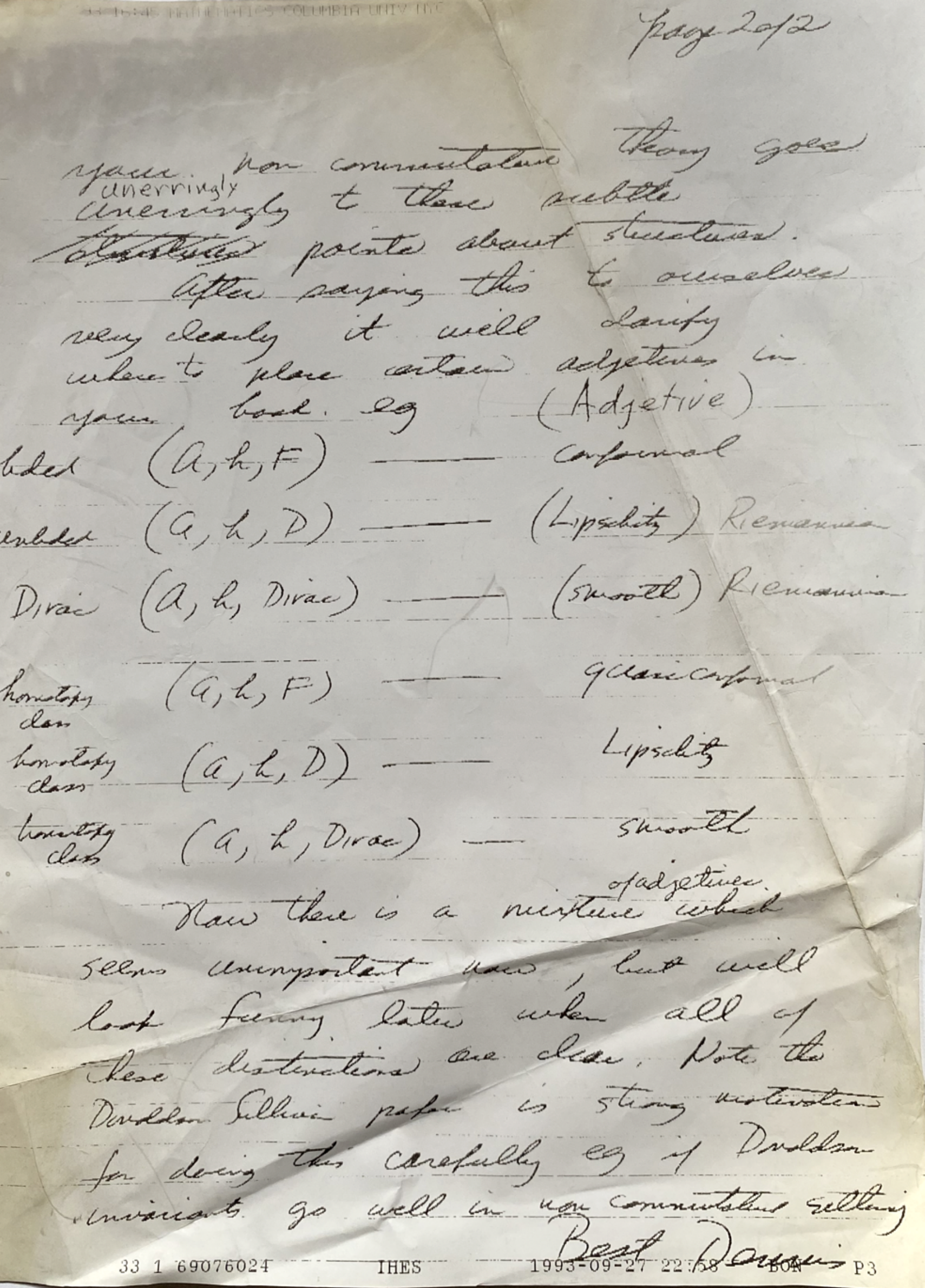}
  \caption{Page 2}
\end{subfigure}
\caption{
"... After lecturing it became clear 	that $(\cA,\cH,F)$ is really the conformal structure, as you always say, and the homotopy class $(\cA,\cH,F')$ is \underline{really} the first instance of the \underline{quasi}-conformal structure... " 
 }
 \label{fax}
\end{figure}
% \begin{center}
\vspace{-.3in}

\begin{tabular}{r | l}
& \\
 Real variable $f:X\to \R$  & Self-adjoint operator $H$ in Hilbert space  \\
Range $f(X)\subset \R$ of the variable & Spectrum of  the operator $H$\\ 
Composition $\phi\circ f$, $\phi$ measurable & Measurable functions $\phi(H)$ of self-adjoint operators  \\
Bounded complex variable $Z$ & Bounded operator $A$ in Hilbert space \\
Infinitesimal variable $dx$ & Compact operator $T$\\
Infinitesimal of order  $\alpha>0$  & Characteristic values $\mu_n(T)=O(n^{-\alpha})$ for $n\to \infty$ \\
 Algebraic operations on functions & Algebra of operators in Hilbert space \\ 
Integral of function $\int f(x)dx$ & $\displaystyle{\int\!\!\!\!\!\! -} T =$ coefficient of $\log(\Lambda)$ in $\Tr_\Lambda(T)$\\
Line element $ds^2=g_{\mu\nu} dx^\mu dx^\nu$ & $ds=\bullet\!\!\!\!-\!\!\!-\!\!\!-\!\!\!-\!\!\!-\!\!\bullet$ : Fermion propagator  $D^{-1}$\\
$d(a,b)=\,{\rm Inf}\,\int_\gamma\,\sqrt{g_{\mu\,\nu}\,dx^\mu\,dx^\nu}$ & $d(\mu,\nu)=\,{\rm Sup}\,\vert \mu(A)-\nu(A)\vert, \mid \ \Vert [D,A]\Vert\leq 1.$ \\ 
Riemannian geometry $(X,ds^2)$& Spectral geometry $(\cA,\cH,D)$\\
Curvature invariants & Asymptotic expansion of spectral action \\
Gauge theory & Inner fluctuations of the metric\\
Weyl factor perturbation & $D\mapsto \rho D\rho$ \\
Conformal Geometry & Fredholm module $(\cA,\cH,F)$, $F^2=1$. \\
Perturbation by Beltrami differential & $F\mapsto (aF+b)(bF+a)^{-1}$, $a=(1-\mu^2)^{-1/2}$, $b=\mu a$ \\
Distributional derivative & Quantized differential $dZ:=[F,Z]$ \\
Measure of conformal weight $p$ & $f \mapsto \displaystyle{\int\!\!\!\!\!\! -}f(Z)\vert dZ\vert^p$ \\
&
\end{tabular}
%\end{center}

\subsection{Schwartz kernels and Schwarzian derivative}\label{Schw&S}
Let $V$ be a 1-dimensional manifold and let $\cH=L^2(V)$ be the Hilbert space of square integrable half densities:  $\xi(x)=f(x) dx^{1/2}\in \cH$. The Schwartz kernel of an operator $T: \cH \to \cH$ is of the form: $k_T(x,y)dx^{1/2}dy^{1/2}$. This means that, as a function of two variables, the kernel depends on choices of positive sections of the 1-dimensional bundle of $1/2$-densities. By varying the choice of a section, i.e. by dividing it with a positive function $\rho(x)$, the kernel $k_T(x,y)$ gets modified accordingly  to  $\rho(x)\rho(y)k_T(x,y)$. Next lemma detects an invariant of the above change

\begin{lem}
 The differential form 
 \begin{equation}\label{qdiff0}\omega=\partial_x\partial_y\log(k_T(x,y))dxdy\end{equation}
  is independent of the choice of sections of the bundle of 1/2-densities and defines an invariant $\omega(T)$ of the operator $T$.
\end{lem}
\proof One sees that by taking the log of the variation of the Schwartz kernel 
$$
\log(\rho(x)\rho(y)k(x,y))=\log(k(x,y))+ \log \rho(x)+ \log \rho(y)
$$
and then applying $\partial_x\partial_y$, the output is independent of  $\rho$. \endproof 
 Let now $V=\R$ and $T=\cF$  the Fourier  transform,  then   $k_{\cF}(x,y)=\exp(-2\pi ixy)$ and the differential form becomes $\omega(\cF)=-2\pi i dxdy$.\newline
Note that the Schwartz kernel of the Fourier transform already appeared in \S \ref{sec:2.3} (in matrix form) and there it played a crucial role in the determination of the K-theory of endomorphisms of $\bars$. 

Next, we compute $\omega(\qd f)$ for the quantized differential of a function $f$ on $\R$.

\begin{lem} Let   $f$ be a smooth, complex valued function on $\R$. Then
\begin{equation}\label{qdiff} \omega(\qd f)=\left(\frac{f'(x) f'(y)}{(f(x)-f(y))^2}-\frac{1}{(x-y)^2}\right)dxdy.
\end{equation}
The restriction of $\omega(\qd f)$ to the diagonal is  $\frac 16\cS(f)dx^2$, where $\cS(f)=\frac{f^{(3)}(x)}{f'(x)}-\frac{3 f''(x)^2}{2 f'(x)^2}$ is the Schwarzian derivative.
\end{lem}
\begin{proof}
The Schwartz kernel of the quantized differential is $k(x,y)=\frac{i}{\pi}\frac{f(x)-f(y)}{x-y}$. By applying  \eqref{qdiff0} one obtains the stated equality \eqref{qdiff}.\end{proof}

Note, in particular, that the second statement of the above lemma shows that the quantized differential determines the Schwarzian  derivative.  

\subsection{The Main Lemma}\label{sectmainlem}
The conceptual reason for the link between Weil's explicit formula and the trace of the compression of the scaling action on Sonin's space \cite{weilpos, quasiinner} is rooted in the following two general facts.

Let $\cH$ be a Hilbert space,  and let $F = 2P-1$ be the operator defining the quantized calculus. An operator $T$ in $\cH$ is encoded by a matrix  $T=\begin{bmatrix} T_{11}&T_{12}\\ T_{21}&T_{22}\end{bmatrix}$ with 
$$
T_{11}=(1-P)T(1-P), \quad T_{12}=(1-P)TP, \quad T_{21}=PT(1-P), \quad T_{22}=PTP
$$

\begin{prop}(\cite{quasiinner} Prop. 5.4) \label{fromqi}
  Let $U=\begin{bmatrix} U_{11}&U_{12}\\0&U_{22}\end{bmatrix}$  be the upper-triangular matrix of an operator in $\cH$. Then $U$ is unitary if and only if the following conditions hold
 \begin{enumerate}
     \item $U_{11}$ is an isometry
     \item $U_{22}$ is a coisometry

\item $U_{12}$ is a partial isometry from $\ker(U_{22})$ to the $\text{coker}(U_{11})$ \end{enumerate}
\end{prop}
Next Lemma (see also  \cite{scalingH} Lemma 3.4) relates the sign of the quantized differential of a triangular unitary operator $U$ to the kernel of the compression $PUP$  of $U$ on $P$ (corresponding to Sonin's space in the application related to Weil's explicit formula).
\begin{lem}\label{mainlemma}
With the  notation of Proposition \ref{fromqi}, let $f$ be a positive operator and $S$ the orthogonal projection to $\ker(U_{22})$. Let  $\tilde S=\begin{bmatrix}0&0\\0&S\end{bmatrix}$. Then
\[
-\frac 12 \text{Tr}(fU^* \qd U)=\text{Tr}(f \tilde S)\ge 0
\]
\end{lem}
\begin{proof}
We first show that $-\frac 12 U^*\qd U=\tilde S$.  We have $$U^*PU=\begin{bmatrix}U_{11}^*&0\\U_{12}^*&U_{22}^* \end{bmatrix}\begin{bmatrix}0&0\\0&1 \end{bmatrix}\begin{bmatrix}U_{11}&U_{12}\\0&U_{22} \end{bmatrix}= \begin{bmatrix}0&0\\0&U_{22}^* \end{bmatrix} \begin{bmatrix}U_{11}&U_{12}\\0&U_{22} \end{bmatrix}= \begin{bmatrix}0&0\\0&U_{22}^*U_{22}\end{bmatrix}.$$ 
Then it follows that $U^*PU-P= \begin{bmatrix}0&0\\0&U_{22}^*U_{22}-1\end{bmatrix}$. One also has $U_{22}^*U_{22}-1=-S$ since $U_{22}$ is a coisometry. Then the  claim follows from the equality $U^*\qd U=2(U^*PU-P)$ together with the fact that the trace of a product of two positive operators is nonnegative.
\end{proof}

\subsection{The semilocal functional equation}\label{semiloc}
In this part we explain the functional equation in the  semilocal case, by  giving a proof of the local functional equation that extends naturally to the semilocal framework. We first introduce some notations.\newline
We denote by $I$ the unitary inversion operator  in the subspace $L^2(\R)^{\text{ev}}$ of even functions in $L^2(\R)$ defined as $I(\xi)(x):=\vert x \vert^{-1}\xi(x^{-1})$. The scaling operator $\rep(\lambda)$, defined for $\lambda\in \rs$, is the unitary operator  in $L^2(\R)^{\text{ev}}$ given by $\rep(\lambda)(\xi)(x)=\lambda^{-1/2}\xi(\lambda^{-1}x)$. One has $I\circ \rep(\lambda)=\rep(\lambda^{-1})\circ I$.
 The Fourier transform $\Fa$is the unitary operator  in $L^2(\R)^{\text{ev}}$ defined by 
 $$
 \Fa(\xi)(y)=\int \xi(x)\exp(-2 \pi i x y) dx.
 $$
 One has $\Fa\circ \rep(\lambda)=\rep(\lambda^{-1})\circ \Fa$. It follows that  $I\circ \Fa$ commutes with the scaling 
 $$
 (I\circ \Fa)\circ \rep(\lambda)= I\circ (\Fa\circ \rep(\lambda))=I\circ (\rep(\lambda^{-1})\circ \Fa)=(I\circ \rep(\lambda^{-1}))\circ \Fa=\rep(\lambda)\circ (I\circ \Fa).
 $$
 The representation $\rep$ of $\rs$ by scaling in $L^2(\R)^{\text{ev}}$ is unitarily equivalent to the multiplication action with the character $\lambda^{-is}$ in $L^2(\R)$, through the Fourier transform $\Fm\circ w$, where $w$ is the unitary isomorphism
 $$
 w:L^2(\R)^{\text{ev}}\to L^2(\R_+^*,dx/x), \quad w(\xi)(\lambda)=\lambda^{1/2}\xi(\lambda)\quad\forall \lambda>0
 $$
 and $\Fm$ denotes the multiplicative Fourier transform 
 $$
 \Fm(f)(s):=\int_0^\infty f(v)v^{-is}d^*v, \ \ d^*v:=dv/v.
 $$
 The von Neumann algebra generated in  $L^2(\R)$ by the multiplications operators  by $\lambda^{-is}$ is equal to $L^\infty(\R)$ acting by multiplication. The Bicommutant Theorem ensures that a unitary operator commuting with this representation is a multiplication operator by a function of modulus one on $\R$. Then it follows that the composite operator in  $L^2(\R)$ 
 \begin{equation}\label{unitaryU}
 (\Fm\circ w)\circ (I\circ \Fa)\circ (\Fm\circ w)^{-1}
\end{equation}
 is the multiplication by a function $u\in L^\infty(\R)$ of modulus one. Next, we shall develop on the following schematic diagram (Figure \ref{schemata})

\begin{figure}[H]
\includegraphics[scale=0.45]{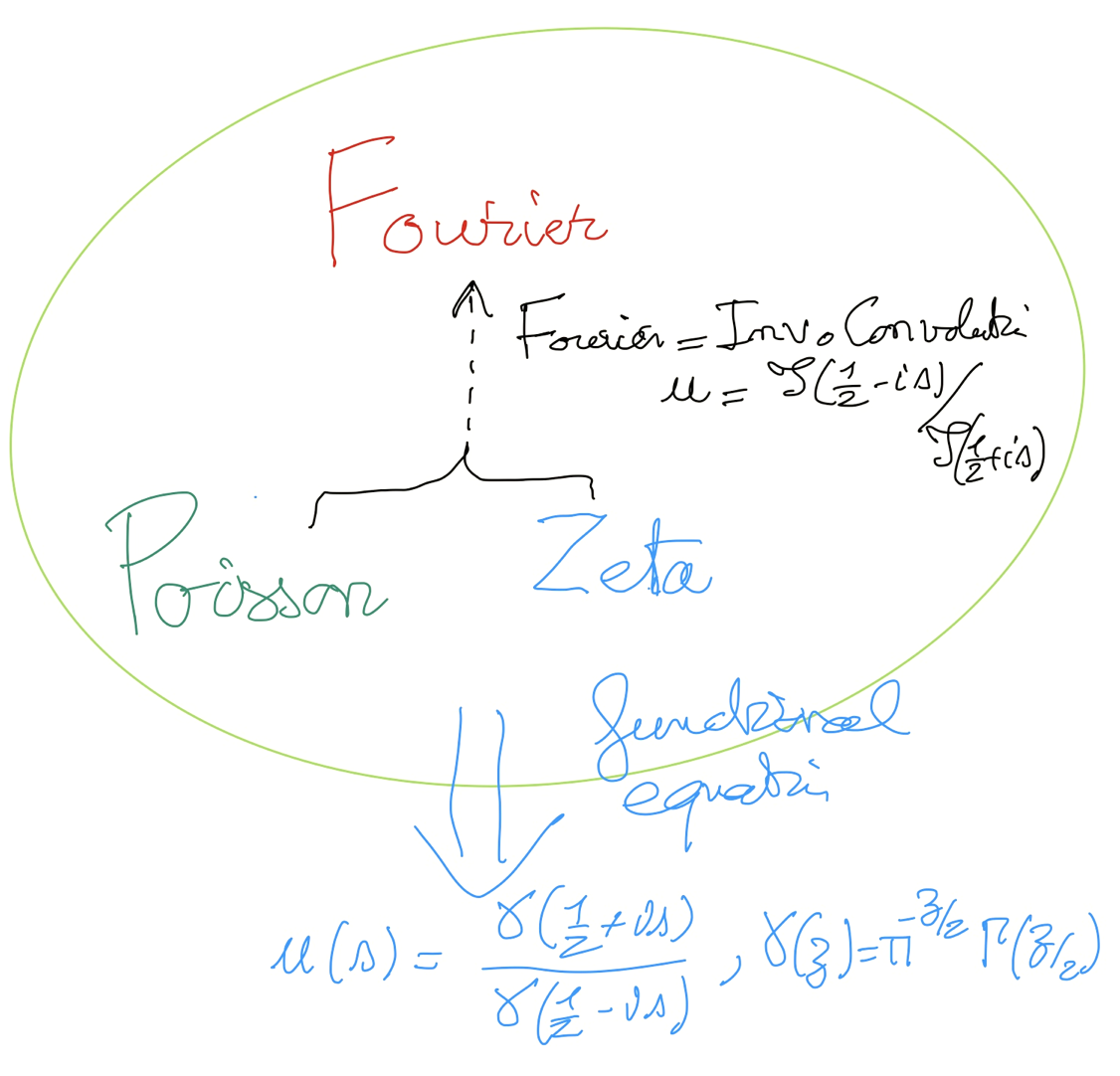}
\centering
\caption{The Poisson formula yields a formula for the unitary $u$ as a ratio of the zeta function on the critical line and the functional equation for the complete zeta function expresses  this ratio as the ratio of local archimedean factors.}\label{schemata}
\end{figure}
\begin{fact}\label{fact1}
The function $u$ is the ratio of local archimedean factors
\begin{equation}\label{unitaryU1}
u(s)=\frac{\zeta(1/2-is)}{\zeta(1/2+is)}= \frac{\gamma(1/2+is)}{\gamma(1/2-is)}, \qquad \gamma(z):=\pi^{-z/2}\Gamma(z/2).
\end{equation}
\end{fact}
\begin{proof}
Let  $f\in \cS(\R)$ be an even function with $f(0)=0=\int f(x)dx$. The Poisson formula ($x>0$) \[
x\sum_{n>0} f(nx)= \sum_{n>0}\hat f(\frac nx)
\] implies
$x^{1/2}\sum_{n>0}f(nx)= x^{-\frac 12}\sum_{n>0} \hat f(n/x)$. Define
 \begin{equation}\label{mapE}\cE(f)(x):=x^{1/2}\sum_{n>0}f(nx)\end{equation} then we have: $
\cE(f)(x)=\cE(\Fa f)(x^{-1})$. 
 By applying the multiplicative Fourier $\F_\mu$ on the right hand side of this equality, one has ($s\in\R$)
\begin{equation}\label{mapE1}
\Fm(\cE(f))(s)=\zeta(1/2-is)\Fm(wf)(s)
\end{equation}
and by the above equality:   $\Fm(\cE(f))(s)=\Fm(\cE(\Fa f))(-s)$. Thus with \eqref{mapE1} one obtains
$$
\zeta(1/2-is)\Fm(wf)(s)=\zeta(1/2+is)\Fm(w \Fa f)(-s)
$$
Finally, in view of \eqref{unitaryU} one writes
$$
(\Fm\circ w)\circ (I\circ \Fa)(f)(s)=\Fm(w \Fa f)(-s)=\frac{\zeta(1/2-is)}{\zeta(1/2+is)}\Fm(wf)(s)
$$
Thus  $u=\frac{\zeta(1/2-is)}{\zeta(1/2+is)}$ is  the ratio of local factors.
\end{proof}

An argument similar to the one just developed in the above proof applies in the semilocal case (when finitely many places are involved, inclusive the archimedean) \cite{scalingH}. More precisely, one lets $S$ be a finite set of places of $\Q$ containing the archimedean place and one considers the semilocal adele class space 
 \begin{equation}\label{XQS}
X_{\Q,S}:=\A_{\Q,S}/\Gamma .
\end{equation}
where $ \A_{\Q,S}=\prod_{v\in S} \Q_v$ is the product of the local fields acted upon (by multiplication) by the subgroup 
\begin{equation}\label{GL1QS}
\Gamma= \{ \pm p_1^{n_1} \cdots p_k^{n_k} \, :\,  p_j
\in S \setminus\{ \infty \} \,,\, n_j\in \Z\}\subset \Q^\times.
\end{equation}
Even though $X_{\Q,S}$ is  a noncommutative space, at the topological level and when $S$ contains at least three places, this space is well behaved at the measure theory level because the additive and multiplicative Haar measures are equivalent on the finite product of local fields \cite{Co-zeta,CMbook}. There is a natural Hilbert space  $L^2(X_{\Q,S})$ of square integrable functions on $X_{\Q,S}$.  Moreover, and very importantly, the Fourier transform $\fourier_\alpha$ on $\A_{\Q,S}$ descends to a unitary operator $\underline\fourier_\alpha$ in $L^2(X_{\Q,S})$. After passing to the dual of the group $C_{\Q,S}=\GL_1(\A_{\Q,S})/\Gamma$ by the Fourier transform $\fourier_C $ and using the inversion $I$, $\fourier_\alpha$ reads as the multiplication by a function $u$ of modulus $1$ on the dual of $C_{\Q,S}$ \cite{CMbook} (Chapter 2)
\begin{equation}\label{FwIPhisemi}
\underline\fourier_\alpha= w^{-1}\circ I \circ \fourier_C   ^{-1}
\circ u \circ   \fourier_C   \circ  w.
\end{equation}
Let $\pi_v$ be  the projection  from the dual of $C_{\Q,S}$ to the dual  $\widehat{\Q_v^\times}$,  the unitary $u$ is of the form 
\begin{equation}\label{RieSiegelsemi}
u= \prod_S\, u_v \circ \pi_v,
\end{equation}
where the $u_v \in L^\infty(\widehat{\Q_v^\times})$ are the functions involved in the local functional equation of Tate \cite{tate}
\begin{equation}\label{PhicF}
\int_{\Q_v^\times} \fourier_{\alpha_v}   (f)(x)\chi(x^{-1})|x|^{1/2}d^*x
 = u_v(\chi)\,\int_{\Q_v^\times} f(x)\chi(x)|x|^{1/2} d^*x.
\end{equation}
Here, $\fourier_{\alpha_v}$ denotes the Fourier transform relative to the canonical additive 
character $\alpha_v$ of the local field $\Q_v$. One knows that the function $u_v$ is 
the ratio of the local factors of $L$-functions. When restricting to the ``zeta sector'' \ie to the subspace of $L^2(X_{\Q,S})$ of functions invariant under the action of the maximal compact subgroup of $C_{\Q,S}$, the function $u$ is the product of ratios of local factors of the Riemann zeta function:  $u=\rho_\infty\prod \rho_p$. This can be proved directly using the same argument as in the above proof of Fact \ref{fact1} (see  \cite{scalingH}). In \cite{quasiinner}, we have developed the notion of quasi-inner function as a generalization of Beurling's notion of inner function which we first related to the main Lemma \ref{mainlemma}. In Theorem 4.1 of \opcit we showed that the product $u=\rho_\infty\prod \rho_p$ of ratios of local factors of the Riemann zeta function  is a quasi-inner function 

\begin{thm}\label{thmquasiinnerintro} The product $u=\rho_\infty\prod \rho_p$ of  ratios of local factors  over a finite set of places of $\Q$ containing the archimedean place is a quasi-inner function relative to $\C_-=\{z\in\C\mid \Re(z)\leq \frac 12\}$.
\end{thm}
This fact shows that the quantized differential $u^*\qd u$ fulfills the hypothesis of the main Lemma \ref{mainlemma} ``modulo infinitesimals" \ie working in the Calkin algebra  quotient of the algebra of bounded operators by the ideal of compact ones. 

\subsection{Quantized calculus and the semilocal trace formula}\label{sectsemiloc}

The key fact enabling the development of the semilocal framework of \S \ref{semiloc} is the equivalence of the additive and multiplicative Haar measures on a finite product of local fields. This fact fails in the global adelic case. The reason why one can go around this difficulty in order to understand the location of the zeros of the Riemann zeta function is that Weil's criterion  
  \begin{equation}\label{weilnegav}
RH \iff \sum_v { W}_v(g*g^*)\leq 0 \qquad \forall g\in C_c^\infty(\R_+^*)\quad \text{with}\quad    \widehat g(\pm \frac i2)=0,
\end{equation}
only involves finitely many primes at a time. Indeed, while  the sum on the right  is extended to all places of $\Q$,  for $v=p$ a non archimedean prime, the functional
 \begin{equation}\label{bombieriexplicit1int}
  W_p(f):=(\log p)\sum_{m=1}^\infty p^{-m/2}\left(f(p^m)+f(p^{-m})\right)
 \end{equation}
vanishes on test functions  with support in the interval $(p^{-1},p)$. Thus
  ${ W}_v(g*g^*)\neq 0$ for only finitely many $v$. The functionals $W_v$   are given by the Riemann-Weil explicit formula
\begin{equation}\label{explicit form}
\widehat f(\frac i2)-\sum_{\frac 12+is \in Z}\widehat f(s)+\widehat f(-\frac i2)=\sum_v { W}_v(f), \quad  \widehat f(s):=\int_0^\infty f(x)x^{-is}d^*x,\quad d^*x=\frac{dx}{x}
\end{equation}
where $Z$ is the multi-set of the non-trivial zeros of the Riemann zeta function.
   The archimedean distribution $W_\R$ is defined as
 \begin{equation}\label{bombieriexplicit2int}
  W_\R(f)=(\log 4\pi +\gamma)f(1)+\int_{1}^\infty\left(f(x)+f(x^{-1})-2x^{-1/2} f(1)\right)\frac{x^{1/2}}{x-x^{-1}}d^*x.
 \end{equation}
The key equality now is  the local trace formula of \cite{Co-zeta} (as revisited in \cite{CMbook})
 \begin{equation}\label{intersection}
 \sum_{v\in S} 	 W_v (f)=\frac 12 \tr(\widehat f  u^*\,\qd u)
 \end{equation}
where the notations for the  right hand side are as in \S \ref{semiloc}.  Lemma \ref{mainlemma} would imply the negativity criterion \eqref{weilnegav}, if  $u$ fulfilled the required hypothesis of that lemma (in \S \ref{semiloc} we pointed out that the failure is only by an infinitesimal). When $S$ is reduced to the single archimedean place,  this difficulty can be bypassed by analyzing the effect of the infinitesimal \cite{weilpos}, and  the expected negativity of the criterion can be derived from Lemma \ref{mainlemma}, where the role of the orthogonal projection to $\ker(U_{22})$ is played by the orthogonal projection  ${\bf S}$ on Sonin's space, one has

\begin{thm}(\cite{weilpos})\label{mainthmintro} Let $g\in C_c^\infty(\R_+^*)$  have  support in the interval $[2^{-1/2},2^{1/2}]$ and Fourier transform  vanishing at $\frac i2$ and $0$. Then  the following inequality holds 
\begin{equation}\label{maininequintronew}
	- W_\R(g*g^*)=W_\infty(g*g^*)\geq \Tr(\rep(g)\, {\bf S}\,\rep(g)^*).
	\end{equation}	
\end{thm} 
In the next \S \ref{radical} we shall see that a difficulty to extend this result in the semilocal case is that the restriction of Weil's quadratic form to test functions with support  in the interval $[\lambda^{-1},\lambda]\subset \R_+^*$ admits (for large values of $\lambda$) extremely small eigenvalues.  This fact prevents  the use of the approximation method developed  in \cite{weilpos}.

\subsection{The radical of the Weil quadratic form}\label{radical}
Weil's quadratic form
\begin{equation}\label{weilquad1}
QW(f,g):=\sum_{1/2+is\in Z} 	\overline{\widehat f(\bar s)}\widehat g(s)
 \end{equation}
admits a non-trivial radical when working with general test functions (\ie if one drops the compact support condition of \eqref{weilnegav}). This radical contains the range of the map $\cE$  defined on the codimension two subspace $\sr0$ of even Schwartz functions fulfilling the equalities $f(0)=0=\widehat f(0)$ by the formula 
\begin{equation}
\cE(f)(x)=x^{1/2}\sum_{n>0}f(nx) \qquad \forall  f\in \sr0.\label{mapeeintro}
\end{equation} 
Indeed,  elements $h=\cE(f)$ of this range fulfill $\widehat h(s)=0$ when $1/2+is\in Z$.\newline
 The Riemann-Weil  explicit formula expresses $QW(f,g)$ on test functions $f,g$ whose support is contained in $[\lambda^{-1},\lambda]\subset \rs$ as a finite sum, thus 
  in a form suitable for numerical testings since the sum involves primes less than  $\lambda^2$. In \cite{ccspectral}    we provided numerical evidence to the fact   that as $\lambda$ increases  the  operator in $\cH(\lambda):=L^2([\lambda^{-1},\lambda],d^*u)$ associated to the quadratic form $QW_\lambda$ (restriction of $QW$) admits  a finite number of extremely small positive eigenvalues. For instance,  we have found that when $\lambda^2=11$  the smallest positive eigenvalue is $2.389\times 10^{-48}$. The presence of these minuscule positive eigenvalues  is explained conceptually by the fact  that the  radical of  Weil's quadratic form contains the range of the  map $\cE$. In \opcit we also gave an excellent approximation of the related eigenfunctions and we  showed that even though  RH implies  $QW_\lambda>0$, (thus that its radical is $\{0\}$),  one can nevertheless  construct, by making use of \eqref{mapeeintro},   functions $g$ with  support in $[\lambda^{-1},\lambda]$ fulfilling: $QW_\lambda(g)\ll \Vert g\Vert^2$. Indeed, let $\cP_\lambda$ and $\widehat{\cP_\lambda}$ be the cutoff projections in the Hilbert space $L^2(\R)^{\rm ev}$, then the projection $\cP_\lambda$ is given by the multiplication with the characteristic function of the interval  $[-\lambda,\lambda]\subset \R$. The projection $\widehat{\cP_\lambda}$ is its conjugate by the Fourier transform $\Fa$. If the even function $f\in \sr0$ belongs to the range of $\cP_\lambda$, then the support of $\cE(f)$ is contained in $(0,\lambda]\subset \rs$. On the other hand, when $f\in \sr0$ is in the range of $\widehat{\cP_\lambda}$ the Poisson formula $
\cE(f)(x)=\cE(\Fa f)(x^{-1})$
 shows that the support of $\cE(f)$ is contained in $[\lambda^{-1},\infty)$. The obstruction to obtain an element $\cE(f)$ in the radical of $QW_\lambda$ is provided by the equality $\cP_\lambda \cap \widehat\cP_\lambda=\{0\} $. The  seminal work of Slepian and Pollack \cite{Slepian, Sl, Slepian0} on band limited functions shows that while $\cP_\lambda \cap \widehat\cP_\lambda=\{0\} $ the prolate functions for small enough eigenvalues almost belong to $\cP_\lambda$ and to $\widehat\cP_\lambda$. Using this fact  we constructed in  \cite{ccspectral} functions denoted    ``prolate vectors'', on which $QW_\lambda$ takes extremely small, non-zero  values.  In the same article we verified concretely that  the orthogonalisation of the prolate vectors give an excellent approximation of the eigenvectors associated to the smallest eigenvalues of Weil's quadratic form.  The first $k+2$ prolate vectors, determine a $k$-dimensional subspace of $L^2( [\lambda^{-1},\lambda],d^*u)\simeq L^2(\rs/\lambda^{2\Z},d^*u)$ on which the associated orthogonal projection  $\Pi(\lambda,k)$ acts. Note that the construction of $\Pi(\lambda,k)$ only uses the prolate vectors  without any reference to $QW_\lambda$.

\subsection{Spectral triples and zeros of zeta} \label{spectrip}
The notion of a spectral triple $(\cA,\cH,D)$ formalizes the concept of a ``spectral geometry'', where the underlying ``space'' is encoded by the algebra  $\cA$ (in general non commutative) that acts by operators in the Hilbert space $\cH$. The self-adjoint  operator $D$ in $\cH$ encodes both the metric aspect of the space (by the formula $d(\mu,\nu)=\,{\rm Sup}\,\vert \mu(A)-\nu(A)\vert,~ \text{with}~ \ \Vert [D,A]\Vert\leq 1$, for the distance between states on the algebra  $\cA$) and the fundamental class in $K$-homology (and also in $KO$-homology if a real structure $J$ is present). \newline
In   \cite{ccspectral}, we have constructed spectral triples   $\Theta(\lambda,k)=(\cA(\lambda),\cH(\lambda),D(\lambda,k))$ making use of the orthogonal projections $\Pi(\lambda,k)$ recalled at the end of \S \ref{radical}.  For this application, the algebra $\cA$ is  $\cA(\lambda):=C^{\infty}(\R_+^*/\lambda^{2\Z})$ acting by multiplication on $\cH(\lambda):=L^2(\R_+^*/\lambda^{2\Z},d^*u)$. The operator $D(\lambda,k)$ is the finite rank perturbation
\begin{equation}\label{iDintro}
D(\lambda,k):=(1-\Pi(\lambda,k))\circ D_0\circ (1-\Pi(\lambda,k)), \ \  D_0=-iu\partial_u
\end{equation}
of the standard Dirac operator $D_0=-iu\partial_u$ (with periodic boundary conditions when viewed in $L^2( [\lambda^{-1},\lambda],d^*u)\simeq L^2(\R_+^*/\lambda^{2\Z})$).
In  Proposition 4.2 of \opcit  we explain  how  we are able to grasp the zeros of the Riemann zeta function up to height $t = 2 \pi  \mu$, 
by computing the spectra of the operators $D(\lambda,k)$ for $\lambda^2\leq \mu$. The computation of the involved prolate vectors 
 only requires the use of integers 
less than the integer part of $\lambda$, because all other  terms in the sum involved in the definition of the map $\cE$ vanish due to the support
condition. This means that we  only use integers $n$ between $1$ and $\lambda$. In this way we have found a remarkable agreement with the first 31 zeros of zeta only implementing the integers $2,3,4$!.\newline
This fact is all the more remarkable since the above restriction on the involved integers (in the sum defining the map $\cE$)  coincides exactly with the restriction in the partial sums occurring in the Riemann-Siegel formula (see 
 \cite{Berry} and \cite{Patterson} \S 6.1). 
 %Theorem, $\vert \Im s\vert=2 \pi XY$, taking $Y=X$ yields $X=\sqrt{s/(2\pi)}$).
 While these findings are largely depending on computer calculations of spectra of large matrices, we have provided their conceptual explanation by introducing  the notion of zeta-cycle \cite{ccspectral}. With $\Sigma_\mu$ denoting the Poincar\' e series operator, \ie the linear map defined on functions $g:\R_+^*\to \C$ by the  formula
\begin{equation}\label{sigmap}
	(\Sigma_\mu g)(u):=\sum_{k\in\Z} g(\mu^ku),
\end{equation}
 We can state the following result, where for $\mu>1$ one lets $C$ be the circle $C:=\R_+^*/\mu^\Z$
  \begin{thm}\label{spectralreal} $(i)$~The spectrum of the action of the multiplicative group $\rs$ on the orthogonal of $\Sigma_\mu \cE(\sr0)$ in $L^2(C)$ is formed of imaginary parts of zeros of zeta on the critical line.\newline
$(ii)$~Let $s>0$ be such that $\zeta(\frac 12+is)=0$, then any circle of length an integral multiple of $2\pi /s$ is a zeta cycle and its spectrum contains $is$.\end{thm}
\begin{rem}
The spectral triples   $\Theta(\lambda,k)=(\cA(\lambda),\cH(\lambda),D(\lambda,k))$ have the same ultraviolet spectral behavior as the Dirac operator $D_0=-iu\partial_u$ on the circle 
$C=\R_+^*/\lambda^{2\Z}$. In particular the number of eigenvalues with absolute value less than $E$ grows linearly with $E$. The ultraviolet behavior of the zeros of the Riemann zeta function is given by  Riemann's formula  for the number $N(E)$ 
of zeros of imaginary part between $0$ and $E$, 
\begin{equation}\label{riembehave}
N(E)=\frac {E}{2\pi}\log \frac {E}{2\pi}-\frac {E}{2\pi} +O(\log E).
\end{equation}
The problem of finding a Dirac operator with the ultraviolet behavior \eqref{riembehave} is solved by the first author  and H.~Moscovici  in \cite{CMoscovici}. Remarkably, the solution involves the prolate spheroidal wave operator $W_\lambda$  whose commutation with the projections 
$\cP_\lambda$ and $\widehat\cP_\lambda$ plays a key role in \cite{Slepian, Sl, Slepian0}.
	
\end{rem}

 \subsection{Prolate vectors  and the semilocal framework}\label{sect37}
The construction of the prolate vectors (and of the projection $\Pi(\lambda,k)$) makes use of the map $\cE$.  In this part we exhibit the relation between this construction  and the semi-local framework, showing that the map $\cE$ appears naturally in the quotient $X_{\Q,S}$ for functions with small enough support.

  \begin{prop}\label{slprol} Let $\mu>1$, $\lambda=\mu^{1/2}$ and $f$ an even function on $\R$ with support in $[-\lambda,\lambda]$. Let $S=\{\infty,2,3,\ldots p'\}$, where $p'$ is the largest prime less than $\mu$. Let $\tilde f$ be the function on $X_{\Q,S}$ associated to the function $ \otimes_{v\in S\setminus \infty} 1_{\Z_v} \otimes f $ where $1_{\Z_v}$ is the characteristic function of the maximal compact subring $\Z_v\subset \Q_v$. Then the following  equality  holds
  \begin{equation}\label{slprol1}
  	 w(\tilde f)(u)=2\cE(f)(u)\qquad \forall u>\lambda^{-1}.
  \end{equation}  	
  \end{prop}
  \proof 
  For $u\in \rs$ one has by construction
  $$
  w(\tilde f)(u)=u^{1/2}\sum_{g\in\Gamma} (\otimes_{v\in S\setminus \infty} 1_{\Z_v} \otimes f)(g(1,1,\ldots,u))
  $$
The terms in the sum vanish unless $g$ is an integer since otherwise $g(1,1,\ldots)\notin \prod \Z_v$. Moreover in that case the terms are  equal to $f(gu)$. Thus the sum becomes 
$$
  w(\tilde f)(u)=u^{1/2}\sum_{g\in \Z\cap \Gamma} f(gu).  $$ 
  Assume   $u>\lambda^{-1}$.  Then for any integer $n$, with $\vert n\vert>\mu =\lambda^2$, one has $f(nu)=0$ since the support of $f$ is in $[-\lambda,\lambda]$. Moreover the following sets are equal since all prime factors of integers less than $\mu$ are in $S$
  $$
  Y=\Z\cap \Gamma\cap [-\mu,\mu]=\{\pm n\mid n\in \N, 0< n\leq \mu\}
  $$
  and thus for  $u>\lambda^{-1}$ one obtains (since $f$ is even)
  $$
  w(\tilde f)(u)=u^{1/2}\sum_{g\in Y} f(gu) =2\cE(f)(u)
   $$ 
   which gives the required equality.\endproof 
For each prime $p$ the characteristic function $1_{\Z_p}$  is its own Fourier transform on $\Q_p$ and this implies that the semi-local Fourier transform $\underline\fourier_\alpha$	 acts as the archimedean Fourier transform $\fourier_{e_\R}$ on functions  $\tilde f$ on the quotient $X_{\Q,S}$,  associated  as above to simple tensors $\otimes 1_{\Z_p} \otimes f$. 
 Thus 
  $$\underline\fourier_\alpha(\tilde f)=\widetilde{\fourier_{e_\R}(f)}. $$
  Moreover if the support of $f$ is contained in the ball of radius $\lambda$ the same holds for $\tilde f$. Together with Proposition \eqref{slprol} this fact suggests that the minuscule eigenvalues of the Weil quadratic form of \S \ref{radical} can be reinterpreted intrinsically in the semilocal framework, without using the map $\cE$, just by analyzing the relative position of the semilocal analogue of the cutoff projections $\cP_\lambda$ and $\widehat{\cP_\lambda}$.

\begin{acknowledgements}
The second author is partially supported by the Simons Foundation collaboration grant n. 691493.
\end{acknowledgements}

\end{document}